\documentclass{amsart}
\usepackage[utf8]{inputenc}

\usepackage{graphics}
\usepackage{thmtools}
\usepackage{wasysym}

\usepackage[T2A]{fontenc}
\usepackage[russian,english]{babel}
\usepackage[utf8]{inputenc}

\usepackage{amsthm}
\usepackage{amsbsy,amsmath,amssymb,amscd,amsfonts}
\usepackage[pagebackref=true]{hyperref}

\usepackage{graphicx,float,latexsym,color}
\usepackage[font={scriptsize,it}]{caption}
\usepackage{subcaption}

\usepackage{makecell}

\usepackage[dvipsnames]{xcolor}

\newtheorem{theorem}{Theorem}
\newtheorem*{theorem*}{Theorem}

\newtheorem{proposition}{Proposition}

\newtheorem{corollary}{Corollary}
\newtheorem{lemma}{Lemma}
\theoremstyle{remark}
\newtheorem{remark}{Remark}
\theoremstyle{definition}
\newtheorem{definition}{Definition}

\hypersetup{
    pdftoolbar=true,        
    pdfmenubar=true,        
    pdffitwindow=false,     
    pdfstartview={FitH},    
    colorlinks=true,       
    linkcolor=OliveGreen,          
    citecolor=blue,        
    filecolor=black,      
    urlcolor=red           
}

\usepackage{lineno}

\usepackage{comment}
\usepackage{microtype}
\usepackage{footnote}

\newcommand{\E}{\mathcal{E}}

\newcommand{\Cm}{\mathcal{C}}
\newcommand{\K}{\mathcal{K}}
\newcommand{\B}{\mathcal{B}}
\newcommand{\T}{\mathcal{T}}

\newcommand{\Sha}{\mathord{\textit{ш}}}

\title{An Infinite, Converging, Sequence of\\ Brocard Porisms}
\author{Dan Reznik} 
\author{Ronaldo Garcia}
\date{September, 2020}

\begin{document}

\maketitle

\begin{abstract}
The Brocard porism is a 1d family of triangles inscribed in a circle and circumscribed about an ellipse. Remarkably, the Brocard angle is invariant and the Brocard points are stationary at the foci of the ellipse. In this paper we show that a certain derived triangle spawns off a second, smaller, Brocard porism so that repeating this calculation yields an infinite, converging sequence of porisms. We also show that this sequence is embedded in a continuous family of porisms.

\end{abstract}

\section{Introduction}
\label{sec:intro}
The Brocard points $\Omega_1,\Omega_2$, shown in Figure~\ref{fig:brocard-basic} are well-studied, unique points of concurrence in a triangle, introduced by August L. Crelle in 1816, given a construction by Karl F.A. Jacobi in 1825, and rediscovered by Henri Brocard in 1875 \cite[Brocard Points]{mw}. 

\begin{figure}
    \centering
    \includegraphics[width=.6\textwidth]{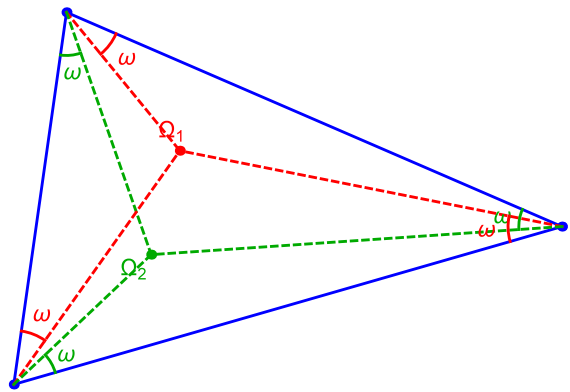}
    \caption{The Brocard Points $\Omega_1$ (resp. $\Omega_2$) are where sides of a triangle concur when rotated about each vertex by the Brocard angle $\omega$. When sides are traversed and rotated clockwise (resp. counterclockwise), one obtains $\Omega_1$ (resp. $\Omega_2$).}
    \label{fig:brocard-basic}
\end{figure}

The Brocard {\em porism}, Figure~\ref{fig:brocard-basic}, is a 1d family of Poncelet 3-periodics (triangles) inscribed in a circle and circumscribed about an ellipse $\E$ known as the {\em Brocard Inellipse} \cite{bradley2007-brocard}. Remarkably, the Brocard angle $\omega$ is invariant {\em and} the Brocard points $\Omega_1$ and $\Omega_2$ are stationary at the foci of $\E$.

As a review, the circle $\K$ which contains both Brocard points and the circumcenter $X_3$ is known as the  Brocard {\em circle} \cite{mw}. The symmedian point $X_6$ also lies on $\K$ and $X_3 X_6$ is known as the Brocard {\em axis}. Notice that over the porism, all said points are stationary, and therefore so is $\K$.

At least seven Brocard {\em triangles}\footnote{The 7th Brocard triangle was proposed during the course of this research.} are known (named 1st, 2nd, etc.) \cite{gibert2020-brocard}, deriving directly from $\Omega_{1,2},\K$ and related objects. It turns out the 1st, 2nd, 5th, and 7th Brocard triangles are inscribed in $\K$; see this \href{http://youtu.be/_bK-BCQv24A}{Video}.

Here we focus on the {\em second} Brocard triangle, denoted $T'$, whose vertices lie at the intersections of symmedians (cevians through the symmedian point $X_6$) with the Brocard circle \cite[Second Brocard Triangle]{mw}; see Figure~\ref{fig:broc-por-sec-tri}.

One first intriguing observation (proved below), unique to $T'$, is that over the porism, its Brocard points $\Omega_1'$ and $\Omega_2'$ of $T'$ are {\em also} stationary; see Figure~\ref{fig:iteration} and this \href{https://youtu.be/MprJtB4UW9s}{Video}. In turn, this leads to a stream of interesting properties.

\subsection*{Main results}

\begin{figure}
    \centering
    \includegraphics[width=\textwidth]{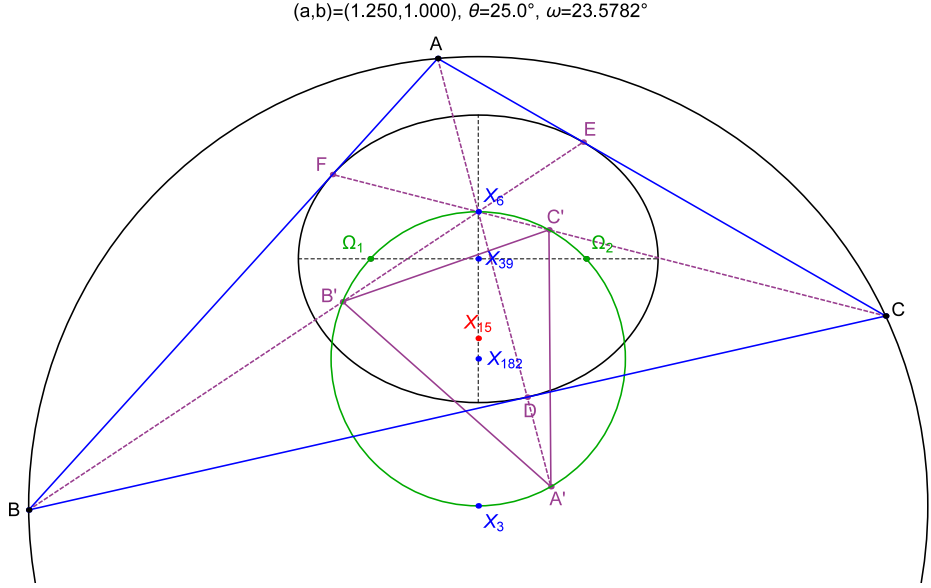}
    \caption{A 3-periodic (blue) $A B C$ in the Brocard Porism is shown inscribed in an external circle $\Gamma$ (black) and the Brocard inellipse $\E$ (black). The tangency points are given by intersections $D,E,F$ of symmedians (cevians through $X_6$) with the sidelengths. The Brocard points $\Omega_1$, $\Omega_2$ are stationary at the foci of $\E$. The Brocard circle $\K$ (green) contains $\Omega_1,\Omega_2$, the circumcenter $X_3$ and the symmedian point $X_6$, all of which are stationary. The center of $\K$ is $X_{182}$, the midpoint of $X_3 X_6$, aka the Brocard axis. Also shown (purple) is the second Brocard triangle $T'=A' B' C'$ inscribed in $\K$ whose vertices lie at the intersections of the symmedians with $\K$. The first isodynamic point $X_{15}$ is on the Brocard axis, is stationary, and common for both the original and the $T'$ family. \href{https://youtu.be/Wgwh4-neJp4}{Video}}
    \label{fig:broc-por-sec-tri}
\end{figure}

\begin{itemize}
   \item The $T'$ are 3-periodics of a second Brocard porism inscribed in $\K$ and circumscribed about a smaller, less eccentric ellipse $\E'$.
   \item Recursive calculation of $T'$ spawns an infinite sequence of ever-shrinking porisms which converge to the first isodynamic point $X{15}$ \cite[Isodynamic Points]{mw}, common to all families. Successive Brocard points lie on two circular arcs.
   \item Recursive calculation of the {\em inverse} 2nd Brocard triangles (denoted $T^*$) converges to a segment between two fixed points $P_2,U_2$ known as the Beltrami points \cite{etc-bicentric}. 
   \item Sequential Brocard circles $\K,\K',\K'',...$ are nested within each other like a Russian doll.
   \item The discrete sequence of porisms is embedded in a continuous family where $T'$ (and its inverse) can be regarded as operators which induce discrete jumps.
   \item The envelope of ellipses in the continuous porism is an ellipse with the isodynamic points as foci.
\end{itemize}


\subsection*{Related Work}

 A construction for the Brocard porism is given in \cite[Theorem 4.20, p. 129]{akopyan2007-conics}. Equibrocardal (Brocard angle conserving) triangle families are studied in \cite{johnson1960}.
Bradley defines several conics associated to the Brocard porism \cite{bradley2011-brocard}. Odehnal has studied loci of triangle centers for the poristic triangle family \cite{odehnal2011-poristic} identifying dozens of stationary triangle centers; similarly, Pamfilos proves properties of the family of triangles with fixed 9-point and circumcircle \cite{pamfilos2020}.  We have previously studied triangle center loci over selected triangle families \cite{garcia2020-brocard,reznik2020-intelligencer}. We have also analyzed the Brocard porism alongside the Poncelet homothetic family establishing a similarity link between the two \cite{reznik2020-similarityII}.

\subsection*{Article Structure} in Section~\ref{sec:review} we review definitions and prove basic Brocard relations required in further sections. In Section~\ref{sec:broc-second} we prove properties of the family of second Brocard triangles in a Brocard porism. In Section~\ref{sec:broc-second} we study the infinite, discrete sequence of porisms induced by iterative calculations of the second Brocard triangle. Finally, in Section~\ref{sec:continuous} we specify how said sequence is embedded in a continuous one.

Appendix~\ref{app:isosceles} provides explicit expressions for the vertices of an isoceles 3-periodic in the porism. Appendix~\ref{app:app_vertices} provides a 1d parametrization for any 3-periodic in the porism. Appendix~\ref{app:app_further} contains additional supporting relations required by some of our proofs. Finally, Appendix~\ref{app:app_symbols} tabulates most symbols used herein.

\begin{figure}
    \centering
    \includegraphics[width=\textwidth]{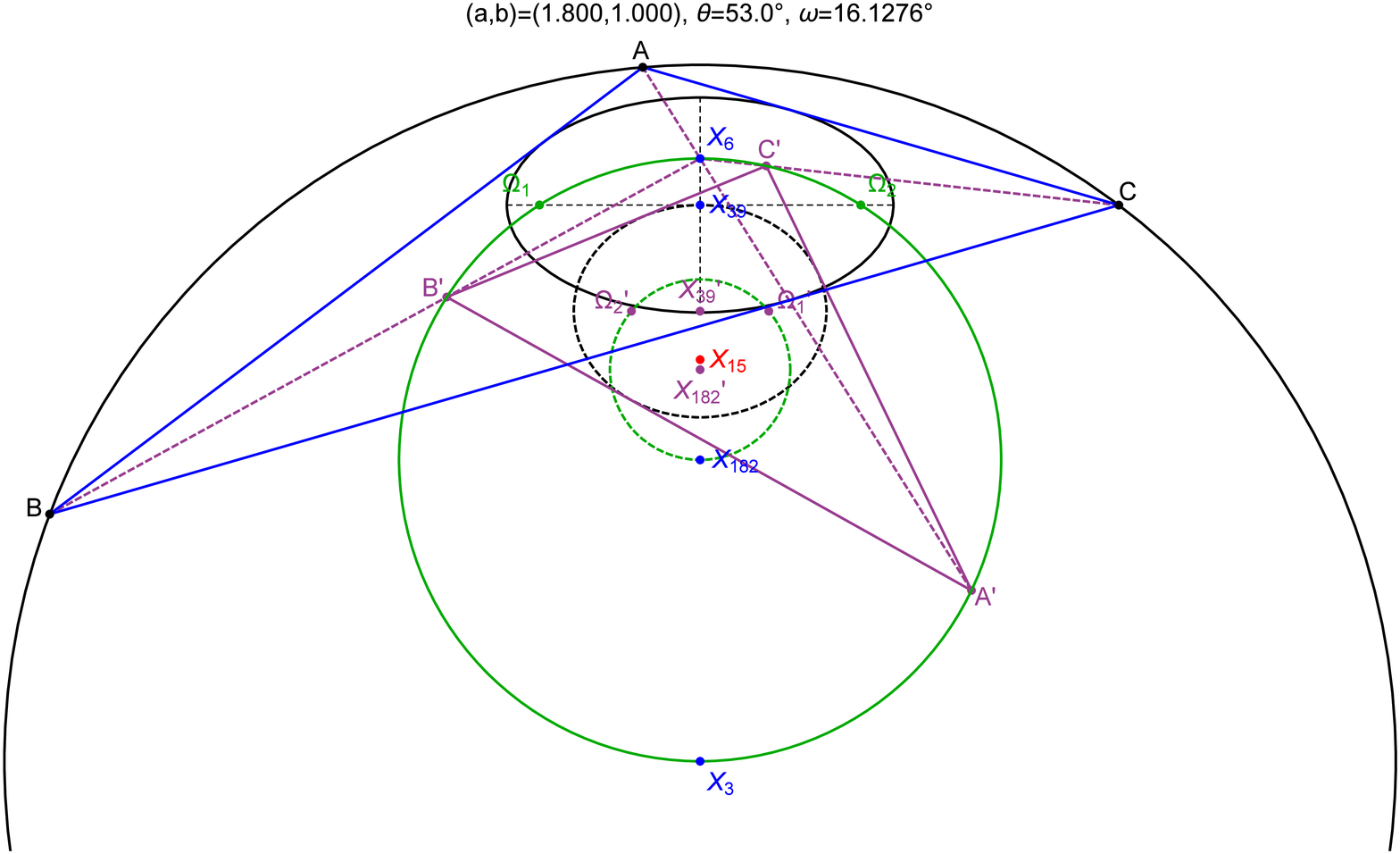}
    \caption{By construction (dashed purple), the second Brocard triangle $T'$ (purple), is inscribed in the Brocard circle $\K$ (green) of the its reference triangle $T$ (blue). Remarkably, over the Brocard porism (outer circle, inner ellipse, both in solid black), the Brocard points $\Omega_{2,1}'$ of $T'$ are also stationary. These coincide with the foci of a new, smaller, rounder, Brocard inellipse $\E'$ (dashed black). The new, stationary, Brocard circle $\K'$ (dashed green) is properly contained within $\K$. Notice the upper vertex of $\E'$ coincides with the Brocard midpoint $X_{39}$ of $T$. The isodynamic points $X_{15},X_{16}$ (latter not shown, above the page) are common and stationary for $T$ and $T'$. \href{https://youtu.be/MprJtB4UW9s}{Video}}
    \label{fig:iteration}
\end{figure}

\section{Properties of Brocard porism triangles}
\label{sec:review}
We adopt Kimberling's $X_k$ notation for triangle centers, e.g., $X_1$ for the incenter, $X_2$ for the barycenter, etc. \cite{etc}. Assume that $X_3=(0,0)$ is at the origin. Referring to Figure~\ref{fig:brocard-basic}

\begin{definition}[Brocard Points]
Let a reference triangle have vertices $ABC$ when traversed counterclockwise. The first Brocard point $\Omega_1$ is where sides $AB$, $BC$, $CA$ concur when rotated a special angle $\omega$ about $A$, $B$, $C$, respectively. The second Brocard point $\Omega_2$ is defined similarly by a $-\omega$ rotation $CB$, $BA$, $AC$ about $C$, $B$, $A$, respectively.
\end{definition}

\noindent Referring to Figure~\ref{fig:broc-por-sec-tri}:

\begin{definition}[Brocard Circle]
This is the circle $\K$ through $X_3$, and the Brocard points $\Omega_1, \Omega_2$.
\end{definition}

\noindent Note $\K$ also contains $X_6$, and it is centered on the midpoint $X_{182}$ of $X_3 X_6$ \cite{mw}.\\

\noindent Referring to Figure~\ref{fig:broc-por-sec-tri}:

\begin{definition}[Brocard Porism]
This is a 1d family of Poncelet 3-periodics inscribed in a circle of radius $R$ (circumcenter $X_3$ is stationary) and circumscribed about an ellipse of semi-axes $(a,b)$ known as the Brocard inellipse $\E$.
\end{definition}

\noindent Over the family, the Brocard angle $\omega$ is invariant and the two Brocard points $\Omega_1$ and $\Omega_2$ are stationary at the foci of $\E$ \cite{bradley2007-brocard}.

\begin{corollary}
Over the Brocard porism the Brocard circle and $X_6$ are stationary.
\end{corollary}

\noindent This stems from the fact that over the porism $\Omega_1,\Omega_2,X_3$ are stationary. Since $X_6$ is antipodal to $X_3$ on the so-called Brocard axis \cite{mw}, it is also stationary.


After \cite{etc}:

\begin{definition}[Barycentric Combo]
Let $P$ and $U$ be finite points on a triangle's plane with normalized barycentrics $(p,q,r)$ and $(u,v,w)$, respectively. Let $f$ and $g$ be homogeneous functions of the sidelengths. The $(f,g)$ {\em combo} of $P$ and $U$, also denoted $f*P + g*U$, is the point with barycentrics $(f\,p + g\,u,f\,q + g\,v, f\,r + g\,w)$.
\end{definition}

\noindent The Cyrillic letter $\Sha$ shall henceforth denote $\cot\omega$. Since $\omega\in[0,\pi/6]$ \cite[Brocard Angle]{mw}, then:

\begin{remark}
 $\Sha{\geq}\sqrt{3}$.
 \label{rem:sha-min}
\end{remark}

\begin{lemma}
Over the porism, the two Isodynamic points $X_{15}$ and $X_{16}$ are stationary and given by:

\[
X_{15}=\left[0, \frac{R(\sqrt{3}-\Sha)}{
 \sqrt{\Sha^2-3}}\right],\;\;\;
  X_{16}=\left[0, -\frac{R(\sqrt{3}+\Sha)}{
 \sqrt{\Sha^2-3}}\right]
\]
\label{lem:x15x16}
\end{lemma}

\begin{proof}
Peter Moses (cited in \cite[X(15), X(16)]{etc}) derives the following combos for the two isodynamic points:

\begin{align}
X_{15} =& \sqrt{3}*X_3 + \Sha*X_6 \label{eqn:combo-x15} \\
X_{16} =& \sqrt{3}*X_3 - \Sha*X_6 \nonumber
\end{align}

With all involved quantities invariant, the result follows. Note that isodynamic points are self-inverses with respect to the circumcircle  \cite[Isodynamic Points]{mw}. For how we obtained the explicit expressions see Appendix~\ref{app:app_further}.
\end{proof}

\begin{lemma}
Let $R$ and $\omega$ denote a triangle's circumradius and Brocard angle. The semi-axes $(a,b)$ and center $X_{39}$ of the Brocard inellipse $\E$ are given by:

\begin{gather*}
[a,b]= R\left[\sin\omega,2\sin^2\omega\right]=R\left[\frac{1}{\sqrt{1+\Sha^2}},\frac{2}{{1+\Sha^2}}\right]\\
 X_{39}=\left[0,-\frac{R\Sha\sqrt{\Sha^2-3}}{\Sha^2+1}\right]
\end{gather*}
 \label{lem:ab}
\end{lemma}

\begin{proof}

Consider a triangle $T$ with sidelengths $s_1,s_2,s_3$, area  $\Delta$, and circumradius  $R$. The following identities appear in \cite{bradley2007-brocard,shail1996-brocard}:

\[R=\frac{s_1 s_2 s_3}{4\Delta}, \;\;\; \sin\omega=\frac{2\Delta}{\sqrt{\lambda}},\;\;\; |\Omega_1-\Omega_2|^2=4c^2=4R^2\sin^2\omega (1-4\sin^2\omega)\]

\noindent where $\lambda=(s_1 s_2)^2+(s_2 s_3)^2+ (s_3 s_1)^2$ (named $\Gamma$ in \cite[Eqn. 2]{shail1996-brocard}), and $c^2=a^2-b^2$. The result follows from combining the above into the following expressions for the Brocard inellipse axes \cite[Brocard Inellipse]{mw}:

\[ a =\frac{s_1 s_2 s_3}{2\sqrt{\lambda}},\;\;\;\;\;\; b =\frac{2 s_1 s_2 s_3 \Delta}{\lambda}.\]

\noindent For how explicit expressions were obtained for $X_{39}$, see Appendix~\ref{app:app_further}.

\end{proof}

\begin{proposition}
The circumradius $R$ and $\Sha$ are given by:

 \[R=\frac{2 a^2}{b},\;\;\; \Sha=\frac{\sqrt{4 a^2-b^2}}{b}.\]
 \label{prop:wRab}
\end{proposition}

\begin{proof}
Follows directly from Lemma \ref{lem:ab}.
\end{proof}

\begin{lemma}
The coordinates for the symmedian point $X_6$ are given by:

\[
X_6=\left[0,-\frac{R\sqrt{ \Sha^2 -3}}{ \Sha}\right] \]
\label{lem:x6}
\end{lemma}

\noindent A derivation is provided in Appendix~\ref{app:app_further}.

\begin{corollary} The distance between circumcenter and symmedian point is given by

\[ |X_3-X_6|=\frac{R\sqrt{ \Sha^2 -3}}{ \Sha}\]
\end{corollary}

\begin{proposition}
In terms of $R$ and $\Sha$, the Brocard points are given by:

 \[\Omega_{1,2}(R,\Sha)=\frac{R\sqrt{ \Sha^2-3}}{ \Sha^2+1} \left[\pm1, - \Sha  \right]\]
 \label{prop:brocs12}
\end{proposition}

\begin{proof} Follows from Proposition \ref{prop:pair_brocard} and Lemma \ref{lem:reciprocal}.
\end{proof}

\begin{corollary}
The center $X_{39}$ of the Brocard inellipse $\E$ is given by:
\[ X_{39}=\left[0,-R {\frac {\Sha\,
\sqrt {\Sha^2-3}}{\Sha^2+1}} \right]\]
\end{corollary}

\section{Properties of the family of second Brocard triangles}
\label{sec:broc-second}
Upwards of seven Brocard triangles are defined in \cite{gibert2020-brocard}. The 1st, 2nd, 5th, and 7th Brocard triangles are inscribed in the Brocard circle, as shown on this \href{https://youtu.be/_bK-BCQv24A}{video}. Henceforth we shall focus on the second Brocard triangle, denoted $T'=A' B' C'$. All primed quantities ($\Omega_i'$, $\omega'$, etc.) below refer to those of $T'$. Specifically, $X_i'$ stands for triangle center $X_i$ of $T'$.\\

\noindent Referring to Figure~\ref{fig:broc-por-sec-tri}:

\begin{definition}[Second Brocard Triangle]
The vertices $A',B',C'$ of the second Brocard triangle $T'$ lie at the intersections of cevians through $X_6$ with the Brocard circle $\K$, i.e., $T'$ is inscribed in $\K$.
\end{definition}

\noindent Over the Brocard porism:

\begin{lemma}
$X_3'$ (equivalent to $X_{182}$) is stationary and given by:

\[
X_3' = X_{182}=\left[0,-\frac{R\sqrt{ \Sha^2 -3}}{2\Sha}\right] \]
\label{lem:x3}
\end{lemma}

This stems from the fact that $T'$ is inscribed in fixed $\K$ whose center is $X_{182}$. The explicit formula is obtained by noting that $X_{182}$ is the midpoint of $X_3 X_6$, with $X_3=[0,0]$ and $X_6$ as given in Lemma~\ref{lem:x6}. 

\begin{lemma}
$X_6'$ (equivalent to $X_{574}$) is stationary and given by:

\[ X_6' = X_{574}= \left[0, -\frac{R\Sha \sqrt{\Sha^2-3}}{\Sha^2+3} \right] \]

\label{lem:x574}
\end{lemma}

\begin{proof}
$X_6'$ is $X_{574}$ of the reference triangle \cite[X(6)]{etc}. The latter is the inverse of $X_{187}$ with respect to the Brocard circle \cite[X(574)]{etc}. In turn, $X_{187}$ is the inverse of $X_6$ with respect to the circumcircle. $X_6$ is stationary in the porism \cite{bradley2007-brocard}. Carrying out the inversions in reverse order, and noting that both the circumcircle and Brocard circle are stationary, obtain the claim.  See Appendix~\ref{app:app_further} for a method to obtain the expression for $X_{574}$.
\end{proof}

\begin{corollary}
 The Brocard Circle $\K'$ of $T'$ is stationary.
\end{corollary}

This stems from the fact that stationary $X_3'$ and $X_6'$ are antipodes on $\K'$ \cite[Brocard Circle]{mw}.

\begin{corollary}
$X_{15}'$ and $X_{16}'$ are stationary.
\label{cor:x15x16p}
\end{corollary}

\begin{proof}
 $X'_{15}$ (resp. $X'_{16}$) coincide with $X_{15}$ (resp. $X_{16}$) \cite[X(15) and X(16)]{etc}, shown in Lemma~\ref{lem:x15x16} to be stationary.
\end{proof}

\begin{corollary}
The $T'$ family is equibrocardal, i.e., $\omega'$ is invariant.
\label{cor:inv-w}
\end{corollary}

\begin{proof}
Plugging invariant $X_3',X_6',X_{15}'$ are stationary in the ``combo'' \eqref{eqn:combo-x15} of Lemma~\ref{cor:x15x16p} yields a unique $\Sha'$. 
\end{proof}

Let $R$ denote the circumradius of a triangle. After \cite[Second Brocard Circle]{mw}:

\begin{definition}[Second Brocard Circle] Let $\K_2$ denote the circle centered on $X_3$ through both Brocard points $\Omega_1,\Omega_2$ and with radius $R_2=R\sqrt{1-4\sin^2\omega}$. 
\end{definition}

\begin{lemma}
The second Brocard circle $\K'_2$ of the $T'$ family is stationary. 
\label{lem:broc2-circ}
\end{lemma}

\begin{proof}
The $T'$ family is inscribed in a fixed circle (i.e., $X_3'$ is stationary and $R'$ is invariant). Since $\omega'$ is invariant (Corollary~\ref{cor:inv-w}), so is $R_2'$, and the result follows.
\end{proof}

\begin{corollary}
The Brocard midpoint $X_{39}'$ of the $T'$ is stationary.
\end{corollary}

\begin{proof}
$X_{39}$ is the inverse of $X_6$ with respect to $\K_2$ \cite[X(39)]{etc}. Since both $X_6'$ and $\K_2'$ are stationary (Lemmas~\ref{lem:x574} and \ref{lem:broc2-circ}), the result follows.
\end{proof}

\begin{corollary}
$\Omega_1'$, $\Omega_2'$ are stationary.
\label{cor:w1w2}
\end{corollary}

\begin{proof}
These both lie on $\K'$ and their join is perpendicular to the Brocard axis $X_3'X_6'$, intersecting it at $X_{39}'$ \cite[Brocard Circle]{mw}. Since all stationary, the result follows.
\end{proof}

\noindent The following results are used to  entail Theorem~\ref{thm:nesting}.

\begin{lemma}
$X_6'$ (equivalent to $X_{574}$) is interior to the segment $X_3'X_6$ (equivalent to $X_{182}X_6$).
\end{lemma}

\begin{proof}
Assume $X_3$ is at the origin. An expression for $X_6$ was given in Lemma~\ref{lem:x6} and one for $X_6'$ in Lemma~\ref{lem:x574}. Noting that $X_{182}=\frac{1}{2}X_6$ and $\Sha^2\geq 3$ (Remark~\ref{rem:sha-min}) yields the result.
\end{proof}

\begin{proposition}
The Brocard circle $\K'$ of $T'$ is contained within its circumcircle, i.e., the Brocard circle $\K$ of its reference triangle.
\label{prop:containment}
\end{proposition}

\begin{proof}
Since $X_3'X_6'$ is a diameter of $\K'$, and both are contained within $\K$ (see Lemma~\ref{lem:x574}), the result follows. 
\end{proof}

\noindent Referring to Figure~\ref{fig:iteration}:

\begin{theorem}
The family of second Brocard triangles are 3-periodics in a new Brocard porism specified by:

\begin{align*}
R'=&   \frac{R\sqrt{\Sha^2-3}}{2\Sha} \\
X_3' =& \left[0,-R'\right] \\
\Sha' =& \frac{\Sha^2+3}{2\Sha}\\
(a',b')=& R'\left(\frac{1}{\sqrt{\Sha'^2+1}},\frac{2}{\Sha'^2+1}\right) \\
\Omega_1'=& \Omega_2\left(R',\Sha'\right) +  X_3' \\
\Omega_2'=&\Omega_1\left(R',\Sha'\right) +  X_3' \\
X_{182}'=&\left[ 0, - \,\frac { 3R'\left(   {\Sha}^{2}+1 \right) }{4\,\Sha'\,\Sha \, }\right]
\end{align*}

\noindent where  $\Sha'=\cot\omega'$ and $\Omega_i(R,\Sha)$,  $i=1,2$ are as in Proposition~\ref{prop:brocs12}.
\label{thm:porism}
\end{theorem}

\begin{proof}
In a general triangle (see Figure~\ref{fig:broc-por-sec-tri}), the major (resp. minor) axes of the Brocard Inellipse are oriented along $\Omega_1\Omega_2$ (resp. the Brocard axis $X_3 X_6$) and its center is the Brocard midpoint $X_{39}$ \cite[Brocard Inellipse]{mw} . Since the Brocard points $\Omega_1'$, $\Omega_2'$ of $T'$ are stationary (Corollary~\ref{cor:w1w2}), the center of $\E'$ is stationary center. Since $X_3$ and $X_6$ are stationary antipodes of $\K$, $\E'$ is axis-aligned with $\E$. Plug invariant $R_2$ and $\omega'$ into the equations in Lemma~\ref{lem:ab} and obtain invariant $(a',b')$ as in the claim.
\end{proof}

Peter Moses let us know that $X_{182}'$ is none other than $X_{39498}$ \cite{moses2020-private-brocard}.

\begin{remark}
The upper vertex of the inellipse of Brocard $\E'$ is at the Brocard midpoint $X_{39}=X_{39}'+[0,b']$.
\end{remark}

\begin{corollary}
The semi-axes of $\E'$ can be expressed in terms of those of $\E$ as follows:
\[ [a',b']= \left[{\frac {a\sqrt {{a}^{2}-{b}^{2}}}{\sqrt {{a}^{2}+2\,{b}^{2}}}},{
\frac {b\sqrt {{a}^{2}-{b}^{2}}\sqrt {4\,{a}^{2}-{b}^{2}}}{{a}^{2}+2\,
{b}^{2}}}\right]\]
\end{corollary}

\begin{proof} This follows from  Theorem \ref{thm:porism} and Lemma~\ref{lem:ab}.
\end{proof}

\section{An infinite sequence of porisms}
\label{sec:porism-seq}

\noindent Referring to Figure~\ref{fig:russian-dolls}:

\begin{theorem}
Recursive calculation of the second Brocard triangle produces an infinite sequence of Brocard porisms $\B',\B'',\B''',...$ such that:
\begin{itemize}
    \item The isodynamic points are stationary at the original $X_{15}$ and $X_{16}$.
\item The Brocard circle of each new porism is contained within the Brocard circle of its parent, forming an infinite nesting.
\item Both the circumradius $R$ and the eccentricity $\varepsilon$ of the inellipse decrease monotonically.
\item The Brocard angle $\omega$ increases monotonically and converges to $\pi/6$ (i.e., triangles approach equilaterals).
\item The sequence of porisms converges to $X_{15}$.
\end{itemize}
\label{thm:nesting}
\end{theorem}

\begin{proof}
Corollary~\ref{cor:x15x16p} entails fixed isodynamic points. Proposition~\ref{prop:containment} entails Brocard circle infinite nesting. The other results stem from recursive application of relations in Theorem~\ref{thm:porism}.
\end{proof}

\noindent Defined in \cite[p. 78]{morley54}, and cited in  \cite[P(2)]{etc-bicentric}:

\begin{definition}[Beltrami Points]
Denoted $P_2$ and $U_2$, these are the inverses with respect to the circumcircle of the Brocard points. They lie on the Beltrami (or Lemoine) axis $L_{563}$, parallel to a line through the Brocard points \cite[Central Lines]{etc}.
\end{definition}

\noindent Introduced in \cite{gibert2020-anti-brocard}:

\begin{definition}(Anti second Brocard triangle)
Given a triangle $T$, its anti second Brocard triangle $T^*$ is a triangle whose second Brocard is $T$. 
\end{definition}

\begin{proposition}
Recursive calculation of the anti second Brocard triangle produces an infinite sequence of Brocard porisms $\B^*,\B^{**},\B^{***},...$ such that:
\begin{itemize}
    \item The isodynamic points remain stationary at the original $X_{15}$ and $X_{16}$.
\item The Brocard circle of each new porism is exterior to the previous one, forming a reverse infinite nesting.
\item The sequence of porisms converges to segment $P_2 U_2$, i.e:
\begin{itemize}
\item The inellipse major (resp. minor) semi-axis converges to $|P_2 U_2|$ (resp. 0).
\item The circumradius $R$ monotonically increases and converges to infinity.
\item The Brocard angle $\omega$ decreases monotonically to zero.
\end{itemize}
\end{itemize}
\label{prop:anti-nesting}
\end{proposition}

\begin{proof}
The result follows applying Theorem \ref{thm:porism}  and Propositions \ref{prop:orbita} and \ref{prop:conjunto_limite}  observing that we can invert the process of recurrence taking the sequence of second anti Brocard triangles  $T^*$, $T^{**},\ldots,$ as   isosceles triangles tangent to the Brocard innelipse as shown in Fig. \ref{fig:russian-dolls-10}.
\end{proof}

\begin{definition}[Beltrami Circles]
Given a triangle, let $\Cm_1$ (resp. $\Cm_2$) denote the first (resp. second) Beltrami circle, passing through the first $P_2$ (resp. second $U_2$) Beltrami point, containing their circumcircle inverse $\Omega_1$ (resp. $\Omega_2$).
\end{definition}

\begin{lemma} The the Beltrami circles $\Cm_1$ and $\Cm_2$ are centered on
\[ P_2,U_2=\left[\mp\frac{R}{\sqrt{\Sha^2-3}}, -\frac{R\Sha}{\sqrt{\Sha^2-3}}\right]
\]
\label{lem:beltrami}
\end{lemma}
\begin{proof} The circumcircle inverse of $\Omega_1$ is $O_1=P_2$ and that of $\Omega_2$ is $O_2=U_2$. 
Therefere, 
\[P_2=R^2\frac{\Omega_1}{|\Omega_1|^2},\;\;\; U_2=R^2\frac{\Omega_2}{|\Omega_2|^2}\]
Therefore the result follows from Proposition \ref{prop:brocs12}. 

\end{proof}

\begin{proposition}\label{prop:beltrami_circles}
The Beltrami circles intersect at $X_{15}$ and $X_{16}$ and their radii $\rho$ is equal and given by:
\[ \rho = \frac{2R}{\sqrt{\Sha^2-3}}\] 
Moreover, the triangles $X_{15} P_2 U_2$ and $X_{16} P_2 U_2$ are equilateral.
\end{proposition}

\begin{proof} Follows from Lemmas \ref{lem:x15x16} and \ref{lem:beltrami}.
\end{proof}

\begin{theorem}
The sequence $\Omega_1$, $\Omega_2'$, ${\Omega_1}''$, ${\Omega_2}'''$, etc. (resp. $\Omega_2$, $\Omega_1'$, ${\Omega_2}'''$, ${\Omega_1}'''$) is concyclic on the first (resp. second) Beltrami circle.
\label{thm:concyclic}
\end{theorem}

\begin{proof} Using Proposition  \ref{prop:orbita}, Lemma \ref{lem:reciprocal}  and Theorem \ref{thm:porism} we compute explicitly the sequence of   Brocrard points stated. It is straightforward to derive the equation of the circles. They are given by

\begin{align*}
  \Cm_1 &:  4 h^2 (3 d^2-h^2)  (x^2+y^2) -8 d h^2 \zeta  x  -4 h (3 d^2+h^2) \zeta y+(3 d^2-h^2) \zeta^2=0\\
    \Cm_2 &: 4 h^2 (3 d^2-h^2)  (x^2+y^2) + 8 d h^2 \zeta  x  -4 h (3 d^2+h^2) \zeta y+(3 d^2-h^2) \zeta^2=0\\
\end{align*}

\noindent where $\zeta=d^2+h^2$. The centers are
\[ O_{1,2}= \left[\pm\frac{d \zeta}{3d^2-h^2}, \frac{(3d^2+h^2)\zeta}{ 2h(3d^2-h^2)}\right]=\left[\pm\frac{R}{\sqrt{\Sha^2-3}} ,-\frac{R\Sha}{ \sqrt{\Sha^2-3}}\right]
\] 
The common radius $\rho$ is given by
\[\rho=\frac{2d \zeta}{3d^2-h^2}=
\frac{2R}{\sqrt{\Sha^2-3}}\]
The intersections of $\Cm_1$ and $\Cm_2$ are  triangle centers $X_{15}$, $X_{16}$ of Lemma \ref{lem:x15x16}.
\end{proof}

\noindent Using the definition in \cite[Isodynamic Points]{mw}:

\begin{definition}(Apollonius Circles)
Given a triangle, there are three circles passing through one vertex and both isodynamic points $X_{15}$ and $X_{16}$.
\end{definition}

Let $\T$ be the upright isosceles 3-periodic with half base $d$ and height $h$.

\begin{corollary}
The Beltrami circles are the first and second Apollonius circles of $\T$. The third Apollonius circle is degenerate and coincides with the Brocard axis.
\end{corollary}

\begin{proof}
The anti-second Brocard $\T^*$ of $\T$ has Brocard points $\Omega_2^*,\Omega_1^*$ which coincide with vertices $B$ and $A$ of $\T$. Applying Theorem~\ref{thm:concyclic} in reverse direction, $\Omega_2^*,\Omega_1^*$ will lie each on $\Cm_2$ and $\Cm_1$, respectively. Therefore $\Cm_2$ (resp. $\Cm_1$) passes through vertex $B$ (resp. $A$) and the two stationary isodynamic points $X_{15}$, $X_{16}$ (anti-Brocards preserve these). By definition these are the first and second Apollonius circles \cite[Isodynamic Points]{mw}. The third Apollonius circle contains the two isodynamic points and vertex $C$ of $\T$. Since these are collinear on the Brocard axis, the circle is a line.
\end{proof}

\begin{proposition}
$\Cm_1$ and $\Cm_2$ are perpendicular to each Brocard circle in the sequence.
\end{proposition}
\begin{proof}
It is straightforward to verify the claim for the circumcircle of isosceles triangle $\T$, given implicitly by
\[ x^2+y^2-R^2=0, \;\; R=\frac{d^2+h^2}{2 h}.\]
\end{proof}
\begin{figure}
    \centering
    \includegraphics[width=\textwidth]{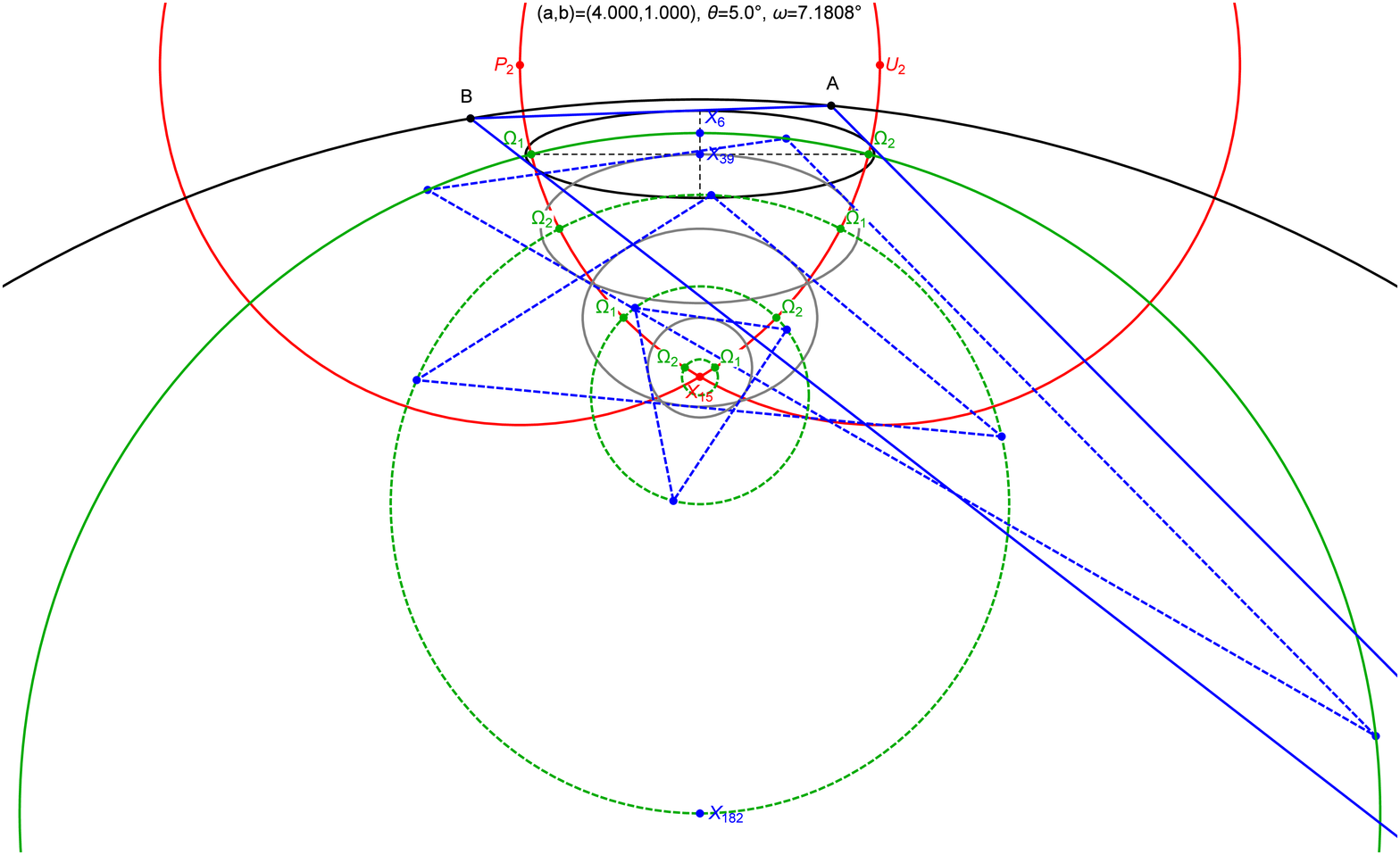}
    \caption{Three iterations of second Brocard triangles (dashed blue), each spawning its own Brocard porism. Successive Brocard points descend in alternate fashion along two circular arcs (red). These intersect at $X_{15}$ and  $X_{16}$ (above page, not shown), with centers on the Beltrami points $P_2,U_2$. The sequence of Brocard points and porisms converges to the first isodynamic point $X_{15}$, common to all porisms. At every generation the circumcircle-inellipse pair approaches a shrinking pair of concentric circles (the triangle family approaches equilaterals).  \href{https://youtu.be/Z3YlEbCFbnA}{Video}}
    \label{fig:russian-dolls}
\end{figure}

\begin{figure}[H]
    \centering
    \includegraphics[width=.9\textwidth]{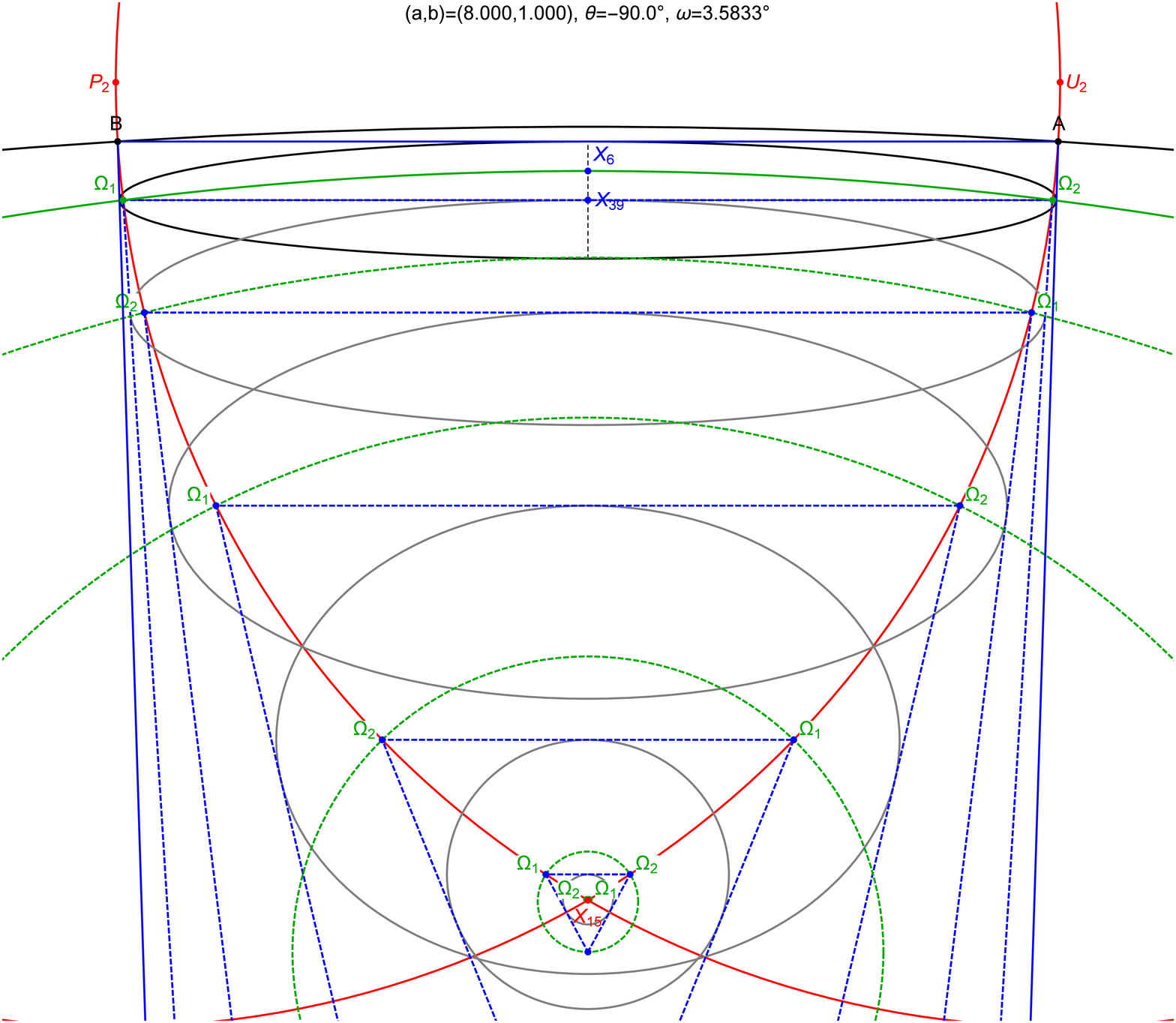}
    \caption{Sequence of porisms with successive Brocard points walking along two circular arcs (red) bounded by the Beltrami points $P_2,U_2$ and $X_{15}$. For each generation the isosceles $\T$ 3-periodic is shown (solid blue = first generation, dashed blue = subsequent ones). The bottom vertex is not shown (below the page). The base vertices $A,B$, or $A',B'$, etc., of a given $\T$ are the Brocard points of the previous generation. And that their midpoint coincides with the top vertex of the Brocard inellipse.}
    \label{fig:russian-dolls-10}
\end{figure}

\section{Embedding the discrete sequence of porisms in a continuous family}
\label{sec:continuous}
Consider a family of axis-aligned ellipses $\E_t$, $0{\leq}t{\leq}\pi/3$ centered at $O_t=[0,-\sin{t}]$ with semi-axes $a,b$ given by:

\[a=\frac{\sqrt{2\,\cos{t} -1}}{2},\;\;\;b=\frac{\sqrt{(2\cos t -1)(1-\cos{t}) }}{\sqrt{2}}\]

\noindent Referring to Figure~\ref{fig:cage}(left):

\begin{remark}
The eccentricity of $\E_t$ is given by $\varepsilon_t= \sqrt{2\cos{t}-1}$. The foci $f_{1,t}$ and $f_{2,t}$ lie each on distinct unit-radius circulars arc centered on $\Cm_1,\Cm_2=[0,\mp{1/2}]$, respectively. Namely:

\[ f_{1,t},f_{2,t}=\left[\cos{t}\pm\frac{1}{2},-\sin{t}\right] \]
\label{rem:foci}
\end{remark}

\begin{theorem}[Continuous]
There is a continuous family of Brocard porisms $\B_t=(\Gamma_t,\E_t)$ whose 1d family of triangles:

\begin{itemize}
\item is inscribed in circle $\Gamma_t$ centered on $X_{3,t}$ with radius $R_t$ given by:
     \[ X_{3,t}=\left[0,\frac{\sin t}{2(\cos t  -1)}\right],\;\;\;R_t=\sqrt{\frac { 2\,\cos t -1}{2(1-\cos t )}} \]
\item circumscribes $\E_t$, its Brocard inellipse (i.e., its Brocard points $\Omega_{1,t},\Omega_{2,t}$ are $f_{1,t}$, $f_{2,t})$ 
\item Has fixed isodynamic points $X_{15,16}$ at $[0,\mp\sqrt{3}/2]$
\item Has fixed Brocard angle $\omega_t = t/2$
\item Has Brocard circle $\K_t$ centered on $X_{182,t}=
[0, \frac{\cos t-2}{2\sin t}]$ with radius $\rho_t=
\frac{2\cos t-1}{ 2\sin t}$
\end{itemize}
\label{thm:continuous}
\end{theorem}

\begin{figure}
    \centering
    \includegraphics[width=\textwidth]{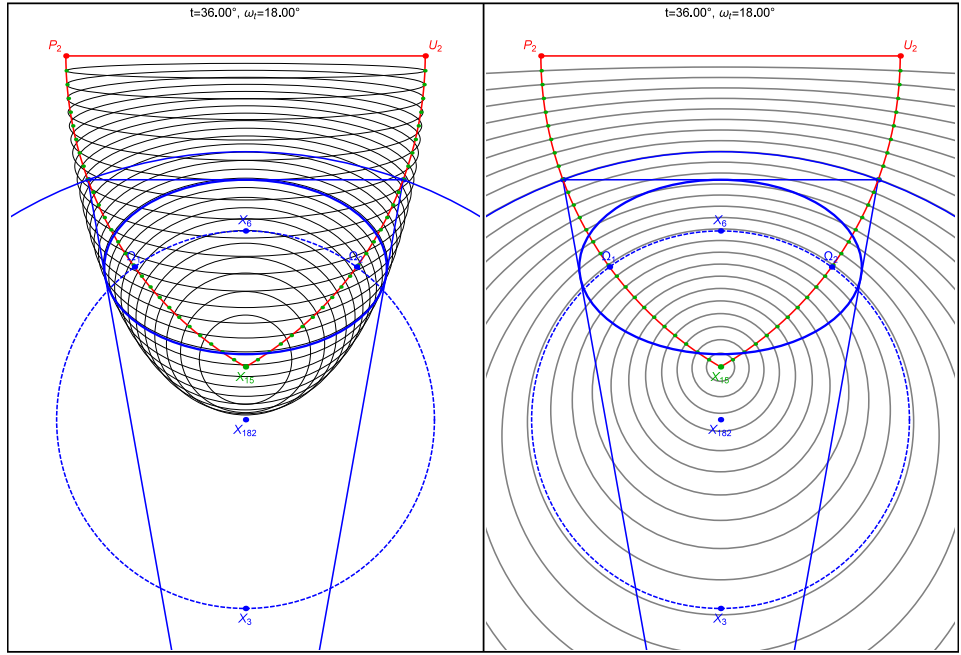}
    \caption{\textbf{Left:} A continuous family of inellipses (gray) is shown with foci (green dots) sliding along two $60^\circ$ circular arcs (red) delimited by $P_2,X_{15}$ and $U_2,X_{15}$. A particular one ($t=36^\circ$) is highlighted (thick blue), with foci $\Omega_1,\Omega_2$. A Brocard porism is formed when $\E_t$ is paired with circle $\Gamma_t$ (blue). An isosceles 3-periodic is shown (blue) interscribed in the pair. Also shown is the fixed Brocard circle (dashed blue). \textbf{Right:} The family of circles $\Gamma_t$ (gray) is nested and converges to $X_{15}$. \href{https://youtu.be/jY_8zxBljuk}{Video}}
    \label{fig:cage}
\end{figure}

\begin{proof}
Let $\T$ be the isosceles triangle defined by the vertices $A=f_1$, $B=f_2$ and $C=[0,\frac{\sin t}{2(\cos t  -1)}]$ ($t$ subscripts are omitted). Let vertex $C$ be at the intersection of the line $U_2 f_2$ with the $y$ axis. Applying Lemma \ref{lem:reciprocal} to this triangle it follows that:
\[R'=  \frac {13\,\cos t-3\,\cos 2t -8  }{4 \sin t },\Sha'=\frac {2-\cos t +2}{\sin t} \]

Invert the expressions for $R',\Sha'$ in Theorem \ref{thm:porism} to obtain the anti Brocard triangle $\T^*$ of $\T$, inscribed in $\Gamma_t$. By construction, $C=X_{3}$ of $\T^*$ and the pair $(\E_t,\Gamma_t)$ is a Brocard porism.

\end{proof}


\begin{remark}
Any Brocard porism is a similarity image of a porism $B_t$ defined in Theorem~\ref{thm:continuous}.
\end{remark}

\noindent Referring to Figure~\ref{fig:detach} (right):

\begin{corollary}[Brocard Circle Containment]
Let $s,u\in[0,\pi/3]$ and $s>u$. Then the corresponding Brocard circles $\K_s\subset\K_u$.
\end{corollary}
\begin{proof}
Combining the the expressions for center and radius of $\K_t$ and in Theorem~\ref{thm:continuous} the claim is true if and only if: 
\[\sin(s-u)+\sin u \geq\sin s.\]
\noindent where $s,u\in [0,\frac{\pi}{3}]$ with $s>u$. 
As the above inequation is equivalent to
\[\sin (s-u)(1-\cos u)+\sin u(1-\cos(s-u))\geq 0\] follows the result.
\end{proof}

 \begin{theorem}[Embedding]
 The family of second Brocard triangles of $\B_t$ are the 3-periodics of a distinct porism $\B_{t'}$, $t'>t$, such that:
 
 \[ \cot{t'}= \Sha' = \frac{4-4\cos{t}+\cos{2t} }{4\sin{t}-\sin{2t}} = \frac{\Sha^2+3}{2\Sha} \]

\label{prop:discrete}
\end{theorem}
 
\begin{proof}
  By Theorem~\ref{thm:continuous}
  the inellipse which yields $\Sha_t$ is given by
    $t=h(\Sha_t)=\tan^{-1}(\frac{2\Sha_t}{\Sha_t^2-1})$.
    From Theorem~\ref{thm:porism} obtain that $\Sha'=g(\Sha_t)=(\Sha_t^2+3)/(2\Sha_t)$. Taking the composition
   $(h\circ g\circ h^{-1})(t)$ leads to the result.
 \end{proof}

\begin{proposition}
 The circles $\K_t$ are perpendicular to $\Cm_1$ and $\Cm_2$
\end{proposition}
\begin{proof} Let 
\begin{align*}
    \Cm_{1,2}:& \;\left(x\pm \frac{1}{2}\right)^2+y^2-1=0\\
    \K_t:&\;x^2+\left(y-\frac{\cos t-2}{2\sin t}\right)^2-\left(\frac{2\cos t-1}{ 2\sin t}\right)^2=0
\end{align*} 

 The intersection of the circles $\Cm_{1,2}$ and $\K_t$
 are the points
 \[
     p_{1,\mp}= \left[  \mp   {\frac {3(2\,\cos t -1)}{2( 5-4\,\cos
t )}},  - {\frac {3\sin t }{ 5-4\,\cos
 t }} \right],\;\;\;
p_{2,\mp}= \left[\mp (\frac{1}{2}-\cos t) , \sin t 
 \right]
\]
Direct calculations show that $\langle \nabla C (p_{i,\mp}),\nabla K(p_{i,\mp})\rangle=0$.
\end{proof}



\begin{remark}
For any $B_t$, the midpoint  $X_{187}$  of $P_2 U_2$ is the inverse with respect to $\Gamma_t$ of the symmedian point $X_{6,t}$.
\end{remark}

Let $Z_t$ denote the lower vertices of $\E_t$. Setting the derivative of $b$ in Theorem~\ref{thm:continuous} to zero obtain:

\begin{corollary}
 The minor semi-axis $b$ (resp. $Z_t$) of $\E_t$ reverses direction of motion at $t_b=\cos^{-1}(3/4){\simeq}41.41^\circ$ (resp. $t_0=\tan^{-1}(4/3){\simeq}53.13^\circ$). Furthermore $0{\leq}b{\leq}1/4$.
\end{corollary}

Note that the major semi-axis $a$ and the y-coordinate of the upper vertex of $\E_t$ are monotonically decreasing.

\begin{figure}
    \centering
    \includegraphics[width=.66\textwidth]{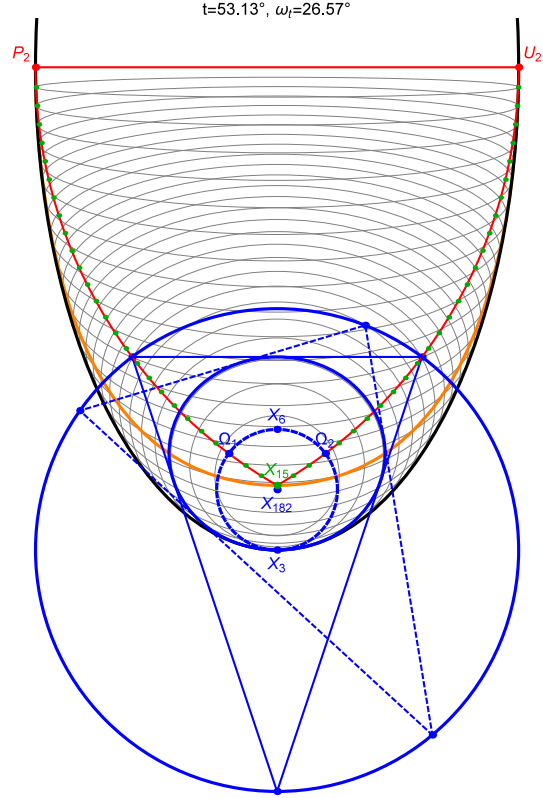}
    \caption{The envelope (thick black) of inellipses $\E_t$ (gray) is an ellipse (only bottom half shown) whose foci are the isodynamic points $X_{15},X_{16}$ common to all families. At $t_0=\cos^{-1}(3/5)$, the porism is such that the Brocard circle (dashed blue) is tangent to $\E_t$ (solid blue), $X_3$ is at the envelope's lower vertex, and the lower vertex of $\E_t$ reverses direction of motion. A non-isosceles 3-periodic (dashed blue) is also shown. The locus (orange) of points at which the 3-web formed by the family of inellipses $\E_t$ is perpendicular to the family of Brocard circles $\K_s$ (taken as independent families) is a quartic containing the isodynamic  and Beltrami points; see Remark~\ref{rem:perp-locus}.}
    \label{fig:detach}
\end{figure}

\begin{proposition}
 The envelope $\xi$ of the family $\E_t$ is given by:
 
 \[ \xi_{1,t},\xi_{2,t}= \left[\pm \frac{ \sqrt{5\cos t-3}}{2\sqrt{\cos t+1}},
 -\frac{2\ sin t}{ \cos t+1}
 \right]
 \]
 and this is contained in the ellipse given by
 \[4x^2+y^2-1=0\]
 Furthermore, the isodynamic points $X_{15},X_{16}$ are at its foci.
\end{proposition}

\begin{proof}
 The ellipse $\mathcal{E}_t$ is parametrized by
 \[ \Gamma(u,t)=[a(t)\cos u, b(t)\sin u]+[0,-sin(t)]\]

The envelope is the solution to $\Gamma(u,t)=\Gamma_t(u,t)=0$ which leads to the claim.
\end{proof}

\begin{remark} At $t>t_0=\cos^{-1}(\frac{3}{5})=\tan^{-1}(\frac{4}{3})$ the family $\E_t$ is nested and so the envelope is empty for $t>t_0.$  See Fig. \ref{fig:detach}.
\end{remark}
  
 \begin{proposition}
When $t<\tan^{-1}(\frac{4}{3})$  the Brocard Circle $\K_t$ intersects the Brocard inellipse $\E_t$ at its envelope $\xi_{1,t}$ and $\xi_{2,t}$.
 \end{proposition}
 \begin{proof}
  The Brocard circle passes through the point $[0,-\frac{\sin t}{2(1-\cos t)}  ]$ and it is tangent to the  Brocard inellipse at the point $[0,-1]$ when $\tan{t}=4/3$.
 \end{proof}

 \begin{proposition}
 The major semi-axis  $a$ is concave and monotonically-decreasing with $a(0)=1, a(\frac{\pi}{3})=0$. The minor semi-axis $b$ is concave with a global maximum $\frac{1}{4}$ attained at $t=\cos^{-1}(\frac{3}{4})$
 and $b(0)=b(\frac{\pi}{3})=0$.
 The eccentricity $\varepsilon(t)=\frac{c(t)}{a(t)} $ is a concave function with $\varepsilon(0)=1, \varepsilon(\frac{\pi}{3})=0$.
 \end{proposition}
 
\begin{proof}
Direct analysis from Theorem \ref{thm:continuous}.
 \end{proof}

Let $V_l(t)\Gamma( -\frac{\pi}{2},t)=[0,-b(t)-\sin{t}]$ denote the lower vertex of $\E_t$.

\begin{corollary}
 $V_l(t)$ moves non-monotonically along the Brocard axis and converges to $X_{15}=[0,-\frac{\sqrt{3}}{2}]$. At $t_0=\tan^{-1}(\frac{4}{5})$ it attains a global minimum $[0,-1]$.
  \label{prop:vertices_ellipse_Et}
\end{corollary}

\begin{proof}  Direct analysis from Theorem~\ref{thm:continuous}.
 \end{proof}

\begin{remark}
 At $t=\tan^{-1}(\frac{4}{5})$,  $\Sha=\frac{5+\sqrt{41}}{4}\approx 2.851$. At $t=\cos^{-1}(\frac{3}{5}) $, $\Sha=2$,
\end{remark}

\begin{remark}
 The family of ellipses $\E_t$ is defined by the implicit differential equation
 \[ 16\,{x}^{2}{y}^{2}\; dx^{2}-8\,xy \left( 4x^2-1 \right){ dx}\,{  dy}+ \left( 16\,{x}^{4}+8\,{x}^{2}+4\,{y}
^{2}-3 \right) dy^{2}=0
 \]
 The family of circles $\K_t$ is given by the differential equation
 \[ 8 x y\; dx-(4x^2-4y^2+3)dy=0\]
 
Consequently, in the interior of the envelope,
the family of circles $\K_t$ and ellipses $\E_t$ define a 3-web \cite{akopyan_2018} which is singular at the coordinate axes and the boundary of the envelope. The orthogonal family to $\K_t$ is the Apollonian
pencil of circles passing through the foci of the envelope. See \cite{akopyan_2018} for the classification of other types of 3-webs defined by circles and ellipses.
\end{remark}
 
\noindent Referring to Figure~\ref{fig:detach}:

\begin{remark}
The families of ellipses $\E_t$ and circles  $\K_s$ (where $t,s$ are independent parameters)
 are orthogonal,
 if and only if,
 \[   16x^4+8x^2+4y^2-3=0.\]
 They are tangent,  if and only if, $x y=0$.
 \label{rem:perp-locus}
 \end{remark}
 


\section{Conclusion}
\label{sec:conclusion}
Above we describe various properties of a discrete sequence of Brocard porisms induced by the second Brocard triangle, and how said sequence is embedded in a continuous one. Videos mentioned in the text appear on Table~\ref{tab:videos}.

\begin{table}[H]
\begin{tabular}{|c|c|c|l|}
\hline
Exp & Title & Link (youtu.be/...)  \\
\hline
01 & \makecell[lt]{1st, 2nd, 5th, and 7th Brocard triangles\\over the Brocard porism} & \href{https://youtu.be/_bK-BCQv24A}{\texttt{\_bK-BCQv24A}} \\
02 & \makecell[lt]{Brocard porism and 2nd Brocard triangle} &  \href{https://youtu.be/Wgwh4-neJp4}{\texttt{Wgwh4-neJp4}} \\
03 & \makecell[lt]{Porism induced by family of 2nd Brocard triangles} &  \href{https://youtu.be/MprJtB4UW9s}{\texttt{MprJtB4UW9s}}  \\
04  & \makecell[lt]{Infinite sequence of Brocard porisms\\induced by the second Brocard triangle} & \href{https://youtu.be/Z3YlEbCFbnA}{\texttt{Z3YlEbCFbnA}} \\
05 & \makecell[lt]{Continuous family of Brocard porisms} &  \href{https://youtu.be/jY_8zxBljuk}{\texttt{jY\_8zxBljuk}} \\
\hline
\end{tabular}
\caption{Illustrative animations, click on the link to view it on {YouTube} and/or enter \texttt{youtu.be/<code>} as a URL in your browser, where \texttt{<code>} is the provided string.}
\label{tab:videos}
\end{table}

\noindent We would like to thank Peter Moses for his prompt and invaluable help with dozens of questions related to Brocard geometry, and Daniel Jaud for his generous and precise editorial help. The second author is fellow of CNPq and coordinator of Project PRONEX/ CNPq/ FAPEG 2017 10 26 7000 508.
\appendix

\section{Isosceles 3-Periodic}
\label{app:isosceles}
Consider an isosceles 3-periodic $\T=ABC$ in the Brocard porism, where $AB$ is tangent to $\E$ at one of its minor vertices. Let $|AB|=2d$ and the height be $h$. Let $\zeta=d^2+h^2$. Let the origin $(0,0)$ be at its circumcenter $X_3$. Its vertices will be given by:

\[A=\left[-d ,\frac{d^2-h^2}{2h}\right], \;\;\; B= \left[d,\frac{d^2-h^2}{2h}\right], \;\;\; \left[0 ,\frac{\zeta}{2h}\right] \]

\begin{proposition}\label{prop:pair_brocard}
The Brocard porism containing $\T$ as a 3-periodic is defined by the following circumcircle $\K_0$ and Brocard inellipse $\E$:

\begin{align*}
\K_0:& x^2+y^2-R^2=0, \;\;\; R=\frac{\zeta}{2h}\\
\E:& -64d^2  h^4  x^2-4  h^2  (9  d^2+h^2)  \zeta  y^2 +4  h  (3  d^2+h^2)  (3  d^2-h^2)  \zeta  y\\&-(d^2-h^2) (9d^2 -h^2) \zeta^2=0
\end{align*}
\end{proposition}

\begin{proof}
The proof follows from $\T$, and isosceles 3-periodic. Recall that the Brocard inellipse is centered at $X_{39}$. Its perspector is $X_6$, i.e., it will be tangent to $\T$ where cevians through $X_6$ intersect it; see Figure~\ref{fig:broc-por-sec-tri}.
\end{proof}

\begin{proposition}
The semi-axes $(a,b)$ of $\E$ are given by:
 
 \[  (a,b)=\left(\frac{d\sqrt{\zeta}}{9d^2+h^2},\frac{4d^2}{9d^2+h^2}\right)
    \]

Furthermore, the Brocard points, located at the foci of $\E$ are given by:

\[ \Omega_1,\Omega_2=\left[ \pm\frac{d(3d^2- h^2)}{9d^2+h^2}, \frac{9d^4-h^4}{2h(9d^2+h^2)}\right] \]
 \label{prop:orbita}
\end{proposition}   
    
    \begin{proof} Follows directly from Proposition \ref{prop:pair_brocard} computing the semi-axes and foci of the ellipse defined by $\E$.
    \end{proof}

\begin{lemma}

\[R= \frac{\zeta}{2h}, \;\;\; \Sha=  \frac{3d^2+h^2}{2dh} ,\;\;\; \sin\omega=\frac{2dh}{\sqrt{(9d^2+h^2)\zeta}}
\]

Or inversely:

\[
d=-{\frac { \left( -2\, \Sha+\sqrt {{ \Sha^2}-3} \right) R}{ \Sha^2
+1}},\;\; h=  \frac{( \Sha^2+\Sha\sqrt{ \Sha^2-3}+3) R}{   \Sha^2+1}
\]
\label{lem:reciprocal}
\end{lemma}
\begin{proof}
Follows from Propositions \ref{prop:pair_brocard} and \ref{prop:orbita}.
\end{proof}

\section{Brocard Porism Vertices}
\label{app:app_vertices}
Let $d,h$ be the half base and height of $\T$, respectively. Let $\zeta=d^2+h^2$.

\begin{proposition}
Parametrics for the 3-periodic vertices in terms of $t$ are given by:

 {\small
 \begin{equation}
     \aligned
     A=& \left[ \frac{\zeta}{2h}\cos t, \frac{\zeta}{2h}\sin{t}\right] \\
     B=&[b_x,b_y]\\
     b_x=&
       -\frac{\zeta  d  \left(2 dh \cos{t}   +(3  d^2+h^2)\sin{t}-3  d^2+h^2\right)}{2  d  h  (3  d^2-h^2) \cos{t}-(9d^4-h^4)\sin{t}+ 9d^4+2d^2h^2+h^4}\\
     b_y=& \frac{\zeta    \left(  2 dh(3  d^2+h^2) \cos{t}      - (9  d^4-2  d^2  h^2+h^4)\sin{t} +9  d^4-h^4 \right)}{2  d  h  (3  d^2-h^2)  \cos{t}-(9d^4-h^4) \sin{t}+ 9d^4+2d^2h^2+h^4}\\
     C=&[c_x,c_y]\\
     c_x=&  -\frac{\zeta d\left(-2dh\cos{t} +(3 d^2+h^2)\sin{t}-3d^2+h^2\right)}{(9d^4-h^4)\sin{t}+2dh(3d^2- h^2)\cos{t}-9 d^4-2 d^2 h^2-h^4}\\
     c_y=&\frac{\zeta   (2    d   h   (3   d^2+h^2) \cos{t} +  (9   d^4-2   d^2   h^2+h^4)\sin{t} -9   d^4+h^4)}{
      2   h   (2   d   h   (3   d^2-h^2)   \cos t+ (9   d^4-h^4) \sin {t}-9   d^4-2   d^2   h^2-h^4) } 
     \endaligned
 \end{equation}
 }
 \end{proposition}
 
 \begin{proof} Given a point $A=R(\cos t,\sin t)\in\K_0$ consider the tangent lines to the Brocard inellipse $\E$ obtained in Proposition \ref{prop:pair_brocard}. The two other intersection points of these lines with $\K_0$ define the vertices $B$ and $C$.
 A long and straightforward calculation  with   help of a CAS leads to the result stated.
 \end{proof}

 

\section{Further Relations}
\label{app:app_further}
Expressions for $X_k$, $k=6,15,16,39,574$ were given above. They were obtained by starting with the explicit calculation of the triangular center $X_i(t)$ over the family of 3-periodic orbits in the Brocard porism given in Proposition \ref{prop:orbita}.

To this end we make use of the trilinear coordinates $f(a,b,c)::$ of  the triangular center $X_i$ \cite{etc} and its conversion to Euclidean coordinates expressed by
  \[X_i = \frac{ s_1 f(s_1,s_2,s_3) A +s_2f(s_2,s_3,s_1)B+ s_3 f(s_3,s_1,s_2) C}
  {s_1 f(s_1,s_2,s_3)+s_2f(s_2,s_3,s_1)+s_3 f(s_3,s_1,s_2)}. \]
  Here $s_1=|B-C|,\; s_2=|C-A|$ and $  s_3=|A-B|$. Obtain the result by simplification via a CA. It is worth mentioning that all triangular centers considered  are rational functions of $(d,h)$. Therefore, by Lemma \ref{lem:reciprocal} they will be rational functions of $(\Sha, \sqrt{\Sha^2-3}).$
 
 \begin{lemma}
 The distance between circumcenter $X_3$ (respec. $X'_3$) and symmedian $X_6$ (respec. $X'_6$) of $T$ (respec. $T'$) is given by

\[ |X_3-X_6|=\frac{R\sqrt{ \Sha^2 -3}}{ \Sha},\;\;\;\;\;
|X_3'-X_6'|=\frac{R(\Sha^2-3)^{\frac{3}{2}}}{2\Sha(\Sha^2+3)}
\]
 \label{lem:dX3X6}
\end{lemma}

\begin{proof} Consider any triangle with circumradius $R$ and Brocard angle $\omega$. By the construction   of the Brocard porism the family   has fixed triangular centers $X_3$,   $X_6$, $X_{182}=\frac{1}{2}(X_3+X_6)$ and $X_{574}$. As $X_3'=X_{182}$ and $X_6'=X_{574}$ the result follows from Lemmas \ref{lem:x3} and \ref{lem:x574}.
\end{proof}

\subsection*{Convergence}

 Recall that given a map $f:U\to U$, $ U\ne \emptyset$, the positive orbit of point $p_0$
 is the set $O_{+}(p_0)=\{p_0,f(p_0), f(f(p_0)),\ldots, f^n(p_0),\ldots\}.$
 Analogously, when $f$ is invertible, the negative orbit
 is defined by
 
 \[ O_{-}(p_0)=\{p_0,f^{-1}(p_0), f^{-1}(f^{-1}(p_0)),\ldots, f^{-n}(p_0),\ldots\}.\]
 
 The future (resp. past) of an orbit is the closure of 
 the positive (resp. negative) orbit and is  denoted by
 the $\omega$-limit set $\omega(p_0)$ (resp. $\alpha$-limit set $\alpha(p_0)$). For an introduction to more properties of these concepts see \cite[Chapter 1]{devaney}.

\begin{proposition}[Convergence]
Let \[ f(R,\Sha)=  (p(R,\Sha),q(\Sha))=\left(  \frac{R\sqrt{\Sha^2-3}}{2\Sha},\;\;\;\;
  \frac{\Sha^2+3}{2\Sha}\right) \]
  
  For any $p_0=[R_0,\Sha_0]$ with $R_0>0$ and $\Sha_0>\sqrt{3} $,  $\omega(p_0)$ is equal to
  $[0,\sqrt{3}]$ and $\alpha(p_0)=[\infty,\infty].$
 \label{prop:conjunto_limite}
\end{proposition}
\begin{proof}
Define $U=\{(R,\Sha) : \Sha\geq \sqrt{3}\}$. We have that $U$ is invariant by $f$, i.e., $f(U)\subset U.$
We observe that for $\Sha>0$,    $q$ has a unique fixed point $(0,\sqrt{3})$.  
 
For any $\Sha_0\geq \sqrt{3}$, $\omega(\Sha_0)=\sqrt{3}$ for $q$ since $q^\prime(\sqrt{3})=0<1$ (attractor). In fact a global attractor. As $\frac{\sqrt{\Sha^2-3}}{2\Sha}<\frac{1}{2}$    for $\Sha>\sqrt{3}$ and $p(x)< \frac{Rx}{2}$ it follows by a graphic analysis that  $\omega( p_0)=[0,\sqrt{3}]$.
 The inverse of $f$ is given by
 \[ f^{-1}(R,\Sha)=\left(2\Sha+\sqrt{\Sha^2-3}, {\frac {\sqrt {2}R\,\sqrt {y+\sqrt {\Sha^{2}-3}}}{\sqrt [4]{\Sha^{2}-3}}}\right)\]
A similar analysis shows that $[0,\sqrt{3}] $ is a repeller of $f^{-1}$ and that the positive orbit of $p_0$ goes to infinity.
 \end{proof}

\section{Table of Symbols}
\label{app:app_symbols}
\begin{table}[H]
\small
\begin{tabular}{|c|l|l|}
\hline
symbol & meaning & note \\
\hline
$T,T'$ & 3-periodic and its 2nd Brocard triangle & \\
$s_1,s_2,s_3$ & sidelengths of $T$ & \\
$\Gamma,R$ & circumcircle and circumradius of $T$ & \\
$\E,a,b$ & Brocard inellipse and semi-axes & centered on $X_{39}$\\
$\B$ & Brocard porism with $\E$ and $\Gamma$ & \\
$\Omega_1,\Omega_2$ & 1st and 2nd Brocard points & on foci of $\E$ \\
$\omega,\Sha$ & Brocard angle and its cotangent & $\Sha=\cot\omega$\\
$\K$ & Brocard circle & centered on $X_{182}$ \\
$\K_2$ & second Brocard circle  & centered on $X_3$ \\
$\T$ & \makecell[lt]{isosceles triangle inscribed in $\Gamma$ \\ whose base is tangent to $\E$ at $(0,b)$} & \\ 
$d,h$ & half-base and height of $\T$ & \\
$\zeta$ & constant $d^2+h^2$ & \\
$P_2,U_2$ & 1st and 2nd Beltrami points & $\Gamma$-inverses of $\Omega_1,\Omega_2$ \\
$\Cm_1,\Cm_2$ & Beltrami Circles &\makecell[lt]{centered on $P_2,U_2$,\\ contain $X_{15}$ and $X_{16}$} \\
\hline
$X_{3}$ & circumcenter (at origin) & $X_3=[0,0]$\\
$X_{6}$ & symmedian point & \\
$X_{13},X_{14}$ & isogonic points & \\
$X_{15},X_{16}$ & isodynamic points & \\
$X_{39}$ & Brocard midpoint & \\
$X_{182}$ & center of $\K$ & midpoint of $X_3 X_6$ \\
$X_{187}$ & Midpoint of Beltrami points & \\
$X_{574}$ & $X_6$ of $T'$ & $\K$-inverse of $X_{187}$\\
$X_{39498}$ & $X_{182}$ of $T'$ & \\
\hline
\end{tabular}
\caption{Symbols used. Primed symbols in the text refer to the second Brocard triangle.}
\label{tab:symbols}
\end{table}

\bibliographystyle{maa}
\bibliography{references,authors_rgk}

\begin{thebibliography}{10}
\expandafter\ifx\csname urlstyle\endcsname\relax
 \providecommand{\url}[1]{doi:\discretionary{}{}{}#1}\else
 \providecommand{\url}{doi:\discretionary{}{}{}\begingroup
  \urlstyle{rm}\Url}\fi

\bibitem{akopyan_2018}
Akopyan, A. (2018).
\newblock 3-webs generated by confocal conics and circles.
\newblock \emph{Geom. Dedicata}, 194: 55--64.
\newblock \url{https://doi.org/10.1007/s10711-017-0265-6}.

\bibitem{akopyan2007-conics}
Akopyan, A.~V., Zaslavsky, A.~A. (2007).
\newblock \emph{Geometry of Conics}.
\newblock Providence, RI: Amer. Math. Soc.

\bibitem{bradley2011-brocard}
Bradley, C. (2011).
\newblock The geometry of the {B}rocard axis and associated conics.
\newblock \url{http://people.bath.ac.uk/masgcs/Article116.pdf}.
\newblock CJB/2011/170.

\bibitem{bradley2007-brocard}
Bradley, C., Smith, G. (2007).
\newblock On a construction of {H}agge.
\newblock \emph{Forum Geometricorum}, 7: 231--–247.
\newblock \url{http://forumgeom.fau.edu/FG2007volume7/FG200730.pdf}.

\bibitem{devaney}
Devaney, R.~L. (2003).
\newblock \emph{An introduction to chaotic dynamical systems}.
\newblock Studies in Nonlinearity. Westview Press, Boulder, CO.
\newblock Reprint of the second (1989) edition.

\bibitem{garcia2020-brocard}
Garcia, R., Reznik, D. (2020).
\newblock Loci of the brocard points over selected triangle families.
\newblock \url{https://arxiv.org/abs/2009.08561}.
\newblock ArXiv:2009.08561.

\bibitem{gibert2020-brocard}
Gibert, B. (2020).
\newblock Brocard triangles and related cubics.
\newblock
  \url{https://bernard-gibert.pagesperso-orange.fr/gloss/brocardtriangles.html}.

\bibitem{gibert2020-anti-brocard}
Gibert, B. (2020).
\newblock Brocard triangles and related cubics.
\newblock
  \url{https://bernard-gibert.pagesperso-orange.fr/gloss/anti-brocardtria.html}.

\bibitem{johnson1960}
Johnson, R.~A. (1960).
\newblock \emph{Advanced Euclidean Geometry}.
\newblock New York, NY: Dover, 2nd ed.
\newblock \url{bit.ly/33cbrvd}.
\newblock Editor John W. Young.

\bibitem{etc}
Kimberling, C. (2019).
\newblock Encyclopedia of triangle centers.
\newblock \url{faculty.evansville.edu/ck6/encyclopedia/ETC.html}.

\bibitem{etc-bicentric}
Kimberling, C. (2020).
\newblock Bicentric pairs.
\newblock \url{bit.ly/2Svh9Dr}.

\bibitem{morley54}
Morley, F., Morley, F.~V. (1954).
\newblock \emph{Inversive Geometry}.
\newblock Chelsea.

\bibitem{moses2020-private-brocard}
Moses, P. (2020).
\newblock Family of triangles with fixed {B}rocard points.
\newblock Private Communication.

\bibitem{odehnal2011-poristic}
Odehnal, B. (2011).
\newblock Poristic loci of triangle centers.
\newblock \emph{J. Geom. Graph.}, 15(1): 45--67.

\bibitem{pamfilos2020}
Pamfilos, P. (2020).
\newblock Triangles sharing their euler circle and circumcircle.
\newblock \emph{International Journal of Geometry}, 9(1): 5--24.
\newblock \url{ijgeometry.com/wp-content/uploads/2020/03/1.-5-24.pdf}.

\bibitem{reznik2020-similarityII}
Reznik, D., Garcia, R. (2020).
\newblock Related by similarity {II}: {P}oncelet 3-periodics in the homothetic
  pair and the {B}rocard porism.
\newblock {arXiv}:2009.07647.

\bibitem{reznik2020-intelligencer}
Reznik, D., Garcia, R., Koiller, J. (2019).
\newblock Can the elliptic billiard still surprise us?
\newblock \emph{Math Intelligencer}, 42.
\newblock \url{rdcu.be/b2cg1}.

\bibitem{shail1996-brocard}
Shail, R. (1996).
\newblock Some properties of {B}rocard points.
\newblock \emph{The Mathematical Gazette}, 80(489): 485–491.

\bibitem{mw}
Weisstein, E. (2019).
\newblock Mathworld.
\newblock \emph{MathWorld--A Wolfram Web Resource}.
\newblock \url{mathworld.wolfram.com}.

\end{thebibliography}

\end{document}